\documentclass[reqno]{amsart}

\usepackage{amssymb}

\def\Comp{\mathop {\fam 0 Comp}\nolimits}
\def\Mor{\mathop {\fam 0 Mor}\nolimits}
\def\chark{\mathop {\fam 0 char}\nolimits}
\def\idd{\mathop {\fam 0 id}\nolimits}

\def\Hom{\mathop {\fam 0 Hom}\nolimits}
\def\End{\mathop {\fam 0 End}\nolimits}
\def\Cend{\mathop {\fam 0 Cend}\nolimits}
\def\Curr{\mathop {\fam 0 Cur}\nolimits}
\def\Derr{\mathop {\fam 0 Der}\nolimits}
\def\codim{\mathop {\fam 0 codim}\nolimits}
\def\Coeff{\mathop {\fam 0 Coeff}\nolimits}
\def\oo#1{\mathrel{{}_{(#1)}}}

\newtheorem{thm}{Theorem}
\newtheorem{prop}{Proposition}
\newtheorem{cor}{Corollary}
\theoremstyle{definition}
\newtheorem{defn}{Definition}
\newtheorem{remark}{Remark}
\newtheorem{exmp}{Example}

\begin{document}

\title[Associative algebras related to conformal algebras]{Associative
algebras related to conformal algebras}

\author{Pavel Kolesnikov}

\thanks{Partially supported by RFBR 05--01--00230,
Complex Integration Program SB RAS (2006--1.9) and
SB RAS grant for young researchers N.~29.
The author gratefully acknowledges the support of the Pierre Deligne fund
based on his 2004 Balzan prize in mathematics.}

\address{Sobolev Institute of Mathematics, Novosibirsk, Russia}

\address{Novosibirsk State University, Novosibirsk, Russia}

\email{pavelsk@math.nsc.ru}

\begin{abstract}
In this note, we introduce a class of algebras that are in some
sense related to conformal algebras. This class
(called TC-algebras) includes
Weyl algebras and some of their (associative and Lie)
subalgebras. By a conformal algebra we generally mean what is known as
$H$-pseudo-algebra over the polynomial Hopf
algebra $H=\Bbbk[T_1,\dots, T_n]$. Some recent results in structure
theory of conformal algebras are applied to get a description
of TC-algebras.
\end{abstract}

\maketitle

\section{Introduction}\label{sec0}

Conformal algebras, initially introduced as an algebraic tool
for studying operator product expansion (OPE) in conformal
field theory, become  objects of pure algebraic study.
By the axiomatic definition appeared in \cite{K1},
a conformal algebra is a linear space $C$ over a field $\Bbbk $
($\chark \Bbbk =0$) endowed with a linear map $T$ and with a family
of operations $(\cdot\oo{n} \cdot ): C\otimes C \to C$,
where $n$ ranges over the set of non-negative integers,
such that:
\begin{itemize}
\item[(C1)] for every $a,b\in C$ only a finite number of $a\oo{n} b $ is nonzero;
\item[(C2)] $Ta\oo{n} b = -n a\oo{n-1} b$;
\item[(C3)] $a\oo{n} Tb = T(a\oo{n} b) + na\oo{n-1} b$.
\end{itemize}

These relations provide a formalization of the following structure.
Suppose $A$ is an algebra (not necessarily associative), and let
$A[[z,z^{-1}]]$ stands for the space of all formal power series over~$A$
in one variable~$z$. Consider
\[
T=\frac{d}{dz}, \quad
(a\oo{n}b)(z) = \mathrm{Res}_{w} a(w)b(z)(w-z)^n,\quad n\ge 0,
\]
where $\mathrm{Res}_{w}x(w,z)\in A[[z,z^{-1}]]$ denotes the coefficient
of $w^{-1}$
in a formal power series $x\in A[[z,z^{-1},w,w^{-1}]]$.
A pair of series $a,b\in A[[z,z^{-1}]]$
is said to be local, if
$a(w)b(z)(w-z)^N =0$ for sufficiently large~$N$.
A space of pairwise mutually local series which is closed
with respect to $T$ and $(\cdot\oo{n}\cdot)$, $n\ge 0$,
is a conformal algebra.

Structure theory of conformal algebras started with \cite{DK,K1,K2},
where the classification of simple and semisimple associative and Lie
conformal algebras was obtained in ``finite'' case,
i.e., if $C$ is a finitely generated $\Bbbk[T]$-module.
In \cite{Z1}, the same result was obtained for Jordan conformal algebras.
In \cite{Ko1,Ko2}, the next step was done: the structure of
some classes of infinite (i.e., infinitely generated
over $\Bbbk[T]$) associative conformal algebras was clarified.

Various features of ordinary algebras can be translated to conformal algebras.
Such a translation is not a ``word-to-word'' one, but there is a general approach
how to get an analogue of some notion (construction) for conformal algebras.
This comes from the fact that ordinary and conformal algebras are
just algebras in certain multicategories corresponding to
Hopf algebras $\Bbbk $ and $\Bbbk[T]$, respectively.

The notion of a multicategory goes back to Lambek \cite{La},
but we will mainly follow \cite{BD} in the exposition.
An algebra in a multicategory is a functor from an operad
$\mathcal C$ to the category.
Given a Hopf algebra $H$, one may construct a multicategory
$\mathcal M^*(H)$ associated with~$H$. An algebra in $\mathcal M^*(H)$
is called an $H$-pseudo-algebra~\cite{BDK}.

In Section~\ref{sec1} we state the general definition of a multicategory.
Section~\ref{sec3} is devoted to the construction
of the category $\mathcal M^*(H)$, where $H$ is a cocommutative Hopf
algebra.
In Section~\ref{sec4} we introduce what is a TC-algebra,
i.e., an ordinary algebra with conformal structure,
and state structure theorems for such algebras.

Throughout the paper, $\Bbbk $ is an algebraically closed field of
characteristic zero.

\section{Multicategories}\label{sec1}

Let us first fix some notation.

Suppose $m\ge n\ge 1$ are two integers.
An $n$-tuple of integers
$\pi=(m_1,\dots ,m_n)$, $m_i\ge 1$,
such that
$m_1+\dots +m_n=m$
we call an $n$-{\it partition} of $m$
 (but we do not assume
$m_i\le m_{i+1}$).
Every partition $\pi $ can be considered as a surjection
(also denoted by $\pi $)
from $\{1,\dots, m\}$ onto $\{1,\dots, n\}$ in the natural way:
\begin{gather*}
\pi (1)=\dots =\pi(m_1) =1, \\
\pi (m_1+1)=\dots =\pi(m_1+m_2) =2, \\
\hdots \\
\pi (m-m_n)= \dots = \pi(m)=n.
\end{gather*}
Therefore, $\pi $ determines one-to-one correspondence
between $\{1,\dots, m\}$ and the set of pairs
$\{(i,j) \mid i=1,\dots, n,\, j=1,\dots, m_i\}$.
This correspondence we will denote by
$\overset{\pi}{\leftrightarrow}$.

By $S_n$, $n\ge 1$, we denote the symmetric group on $\{1,\dots ,n\}$.
Suppose $\pi=(m_1,\dots ,m_n)$ is a partition,
$\sigma \in S_n$.
Then $\sigma \pi=(m_{\sigma^{-1} (1)},\dots ,m_{\sigma^{-1} (n)})$
is again a partition. Thus, $S_n$ acts on the set of all $n$-partitions
of~$m$.

Given an $n$-partition $\pi $ of $m$, $\sigma \in S_n$,
and a family $\tau_i\in S_{m_i}$ ($i=1,\dots ,n$),
one may define $\sigma ^\pi (\tau_1,\dots ,\tau_n)\in S_m$
in the following way:
if $k\overset{\pi}{\leftrightarrow}(i,j)$
then
\[
\sigma ^\pi (\tau_1,\dots ,\tau_n)(k)
 \overset{\sigma \pi}{\leftrightarrow} (\sigma (i),\tau_i(j)).
\]
This composition rule turns the collection of $S_n$, $n\ge 1$,
into an operad, known as the operad of symmetries.

Also, define what is to be a composition of two partitions:
if $p\ge m\ge n\ge 1$,
$\tau = (p_1,\dots ,p_m)$ is an $m$-partition of~$p$,
$\pi=(m_1,\dots ,m_n)$ is an $n$-partition of~$m$,
then there exists
an $n$-partition of $p$, denoted
$\pi\tau = (q_1,\dots ,q_n)$,
where
$q_1=p_1+\dots p_{m_1}$,
$q_2=p_{m_1+1} + \dots +p_{m_1+m_2}$, etc.

\subsection{Definition of a multicategory}

The following definition is due to \cite{La}. A similar notion was
defined in \cite{BD} as {\it pseudo-tensor category} (see also
\cite{BDK}). We will mainly follow \cite{Le} in the exposition.
Another approach to (essentially) same structures was developed in
a recent book~\cite{Fi}.

\begin{defn}\label{defn1}
Let $\mathcal A$ be a class of objects
such that

\begin{itemize}
\item[(M1)]
for any integer $n\ge 1$
and for any family of objects $A_1,\dots, A_n, A\in \mathcal A$
there exists a set
$P^{\mathcal A}_n(A_1,\dots, A_n; A)$
(simply, $P_n(\{A_i\}; A)$)
of $n$-{\it morphisms};

\item[(M2)]
for any $A_1,\dots ,A_m\in \mathcal A$,
$B_1,\dots ,B_n\in \mathcal A$,
$C\in \mathcal A$,
and for any partition $\pi =(m_1,\dots ,m_n)$
of~$m$
we are given a map
\begin{equation}\label{Comp}
\Comp^\pi :
P_n(\{B_i\}; C)
 \times  \prod\limits_{i=1}^n
P_{m_i}(\{A_{ij}\}; B_i)
\to P_m(A_1,\dots A_m; C),
\end{equation}
We will denote
\[
 \Comp^\pi (\varphi ,\psi_1,\dots, \psi_n)
=\Comp^\pi \big(\varphi ,(\psi_i)_{i=1}^n\big)
= \varphi (\psi_1,\dots ,\psi _n).
\]
Such a map $\Comp $ is called a {\it composition map};

\item[(M3)]
there is a symmetric group action on $n$-morphisms, i.e.,
\begin{eqnarray}\label{Symm}
\sigma: && P_n^{\mathcal A}(A_1, \dots , A_n; A) \to
 P^{\mathcal A}_n (A_{\sigma^{-1} (1)}, \dots , A_{\sigma^{-1} (n)}; A), \\
&&f\mapsto \sigma f, \nonumber
\end{eqnarray}
such that 
$(\sigma\tau)f = \tau(\sigma f)$, $\sigma,\tau \in S_n$.
\end{itemize}

If these structures satisfy the following axioms,
then $\mathcal A$ is said to be a {\it multicategory}.

\begin{itemize}
\item[(A1)]
The composition map is associative.
Namely, suppose we have
three families of objects
$A_h\in \mathcal A$ ($h=1,\dots ,p$),
$B_j\in \mathcal A$ ($j=1,\dots ,m$),
$C_i\in \mathcal A$ ($i=1,\dots ,n$),
a partition
$\tau=(p_1,\dots ,p_m)$ of~$p$,
and a partition
$\pi = (m_1,\dots ,m_n)$ of $m$.
Also, let  $D$ be an object of $\mathcal A$, and
\[
 \psi_j\in P_{p_j}(\{A_{jt}\}; B_j),\quad
\chi_i \in P_{m_i}(\{B_{it}\}; C_i),\quad
\varphi \in P_n(\{C_i\}; D)
\]
be multimorphisms of $\mathcal A$.
Then
\begin{multline}   \label{P-assoc}
\Comp^\tau \big(\Comp^\pi \big(\varphi ,(\chi_i)_{i=1}^n\big),
 (\psi_j)_{j=1}^m \big) \\
=
\Comp^{\pi\tau}\big(\varphi ,
\big(\Comp^{\tau_i} \big(\chi_i, (\psi_{it})_{t=1}^{m_i} \big)_{i=1}^n\big)
\big),
\end{multline}
where $\tau_i = (p_{i1},\dots ,p_{im_i})$ is the ``subpartition'' of $\tau$.

\item[(A2)]
For any
$A\in \mathcal A$ there exists
a ``unit", i.e.,  1-morphism
$\idd_A \in P_1(A; A)$
such that
\begin{equation}  \label{P-unit}
\Comp^{\idd(n)} (f, \idd_{A_1},\dots , \idd_{A_n}) =
\Comp^\varepsilon (\idd_A, f) = f
\end{equation}
for any $f\in P_n(\{A_i\}; A)$; here $\idd(n) $
is the identity partition $(1,\dots ,1)$,
$\varepsilon $ is the trivial one: $\varepsilon =(n)$.

\item[(A3)]
The composition map is equivariant with respect to the symmetric group
action, i.e.,
if $\sigma \in S_n$,
$\pi=(m_1,\dots ,m_n)$
is a partition of $m$,
$\tau_i\in S_{m_i}$ ($i=1,\dots ,n$),
$\psi_i \in P_{m_i}(\{A_{ij}\},B_i)$,
$\varphi \in P_n(\{B_i\}, C)$,
then
\begin{equation}\label{equivar}
\Comp^{\sigma \pi }
\big(\sigma \varphi, (\tau_{\sigma^{-1} (i)}\psi_{\sigma^{-1}(i)})_{i=1}^n \big)
=
\sigma^\pi(\tau_1,\dots ,\tau_n)
\Comp^\pi
\big( \varphi, (\psi_{i})_{i=1}^n \big).
\end{equation}
\end{itemize}
\end{defn}

Suppose $\mathcal A$ and $\mathcal B$ are two multicategories.
A {\it functor} $F: \mathcal A \to \mathcal B$
is a rule $A\mapsto F(A)$, $A\in \mathcal A$, $F(A)\in \mathcal B$,
such that for any
$\varphi \in P_n(\{A_i\}; A)$
there is $F(\varphi )\in P_n(\{F(A_i)\}; F(A))$
and
\begin{itemize}
\item $F$ preserves composition maps, i.e.,
\[
F\big( \Comp^\pi\big(\varphi ,(\psi_i)_{i=1}^n\big)\big)
= \Comp^\pi\big(F(\varphi ), (F(\psi _i))_{i=1}^n\big);
\]
\item $F(\idd_A) = \idd_{F(A)} $;
\item $F(\sigma \varphi )=\sigma F(\varphi )$, $\sigma \in S_n$.
\end{itemize}

A multicategory $\mathcal A$ can be considered as an
ordinary category with respect to
$\Mor^{\mathcal A}(A,B) = P_1^{\mathcal A}(A;B)$.

It is suitable to distinguish a ``linear version" of a
multicategory, assuming
\begin{itemize}
\item
all $P^{\mathcal A}_n(\{A_i\}; A)$ are linear spaces over a ground field $\Bbbk$;
\item
the composition map is polylinear;
\item
symmetric group action is linear.
\end{itemize}
In the linear case, functors are supposed to be
linear maps of spaces of multimorphisms.

\subsection{Algebras in a multicategory}

One of the most natural examples of a multicategory is the
category $\mathcal V_\Bbbk $ of linear spaces over  a field $\Bbbk $,
where $P_n^{\mathcal V_\Bbbk}(A_1,\dots, A_n; A)$
consists of all polylinear maps from
$A_1\times \dots \times A_n$ to $A$.

A multicategory with only one object
is known as an {\it operad} (see, e.g., \cite{Le}).
If $\mathcal C$ is a (linear) operad, $C$ is the object of
$\mathcal C$, then the multicategory structure on $\mathcal C$
is completely determined by the collection of spaces
$\mathcal C(n) = P^{\mathcal C}_n(C,\dots, C; C)$
endowed with the structures (M1)--(M3).

Let us state an example of an operad that will be used later.
Suppose $X=\{x_1,x_2,\dots \}$ is a countable
set of variables, and  $\Bbbk\{X\}$ is the free
(non-associative) algebra generated by $X$ over $\Bbbk $.

\begin{exmp}\label{exmp1}
Let $I$ be a collection of homogeneous polylinear polynomials in $X$.
Denote by $(I)$ the T-ideal of $\Bbbk\{X\}$ generated by $I$.
Then $A=\Bbbk\{X\}/(I)$ is the free algebra in the corresponding
variety.
Define
$\mathcal C_I(n)$
to be the $\Bbbk$-linear span of images in $A$
of all non-associative words of length $n$ in $x_1,\dots ,x_n$.

The action of $S_n$ on $\mathcal C_I(n)$
is just the permutation of variables
(it is well-defined), and the composition $\Comp^\pi $ is the
substitution with relabelling variables (also well-defined):
\begin{multline}
\Comp^\pi:
(f(x_1,\dots ,x_n), g_1(x_1,\dots x_{m_1}), \dots g_n(x_1,\dots ,x_{m_n})
\mapsto \\
f(g_1(x_{11},\dots ,x_{1m_1}), g_2(x_{21},\dots ,x_{2m_2}),\dots
g_n(x_{n1},\dots x_{nm_n})),  \nonumber
\end{multline}
$\pi=(m_1,\dots ,m_n)$.
Here we identify $x_{ij}$ with $x_k$ via the
partition~$\pi $ as above.
Note that any multimorphism $f\in \mathcal C_I(n)$ can be obtained
(up to permutation of variables) as a composition
of $\idd\in \mathcal C_I(1)$ and
$\mu = \overline{x_1x_2}\in \mathcal C_I(2)$.
\end{exmp}

\begin{defn}[\cite{BD}]\label{defnBD}
Let $\mathcal A$ be a multicategory, and let $\mathcal C$
be an operad. A $\mathcal C$-{\em algebra\/} in $\mathcal A$
is a functor $F:\mathcal C\to \mathcal A$.
\end{defn}

This definition allows to consider the class of
 $\mathcal C$-algebras
as a category, but not as a multicategory.
We propose slightly different definition in order
to make the class of algebras to be a multicategory.

Suppose $\mathcal A$ and $\mathcal B$ are two multicategories,
and let
\[
F,G: \mathcal A \to \mathcal B
\]
be two functors.

\begin{defn}\label{defn3}
Consider the class of objects $\mathcal P(\mathcal A, F, G)$
that consists of pairs
$(A, \mu )$,
where
$A\in \mathcal A$,
$\mu\in P_1^{\mathcal B}(F(A); G(A))$.
Define multimorphisms on $\mathcal P(\mathcal A, F, G)$
as those $f\in \mathcal P_n^{\mathcal A}(A_1,\dots A_n; A)$
that satisfy the relation
\begin{equation}         \label{eq:alg}
G(f)(\mu _1,\dots ,\mu _n)
 =
\mu(F(f))\in P^{\mathcal B}_n(\{F(A_i)\}; G(A)),
\end{equation}
where $(A_i,\mu _i)$, $(A,\mu )$ are objects of
$\mathcal P(\mathcal A, F, G)$.
The structure obtained is a multicategory of
$(F,G)$-{\it pseudo-algebras} in~$\mathcal A$.
\end{defn}

For example, if $\mathcal A=\mathcal B =\mathcal V_{\Bbbk}$,
$F(A)=A\otimes A$, $G(A)=A$ then an $(F,G)$-pseudo-algebra in
$\mathcal A$ is just an ordinary algebra over a field.
It is easy to see that in this case
$\mu \in P_2^{\mathcal V_{\Bbbk}}(A,A;A)$,
so there exists a functor $\Phi: \mathcal C_{\emptyset} \to
\mathcal V_{\Bbbk}$, defined by
$\Phi(\overline{x}_1)=\idd_A$,
$\Phi(\overline{x_1x_2})=\mu$.
Therefore, this is also an algebra in the sense of Definition~\ref{defnBD}.
Later we will consider another example of a multicategory,
that leads to the notion of a conformal algebra.

\section{$H$-pseudo-algebras and $H$-conformal algebras}\label{sec3}

In this section, we generally follow~\cite{BDK}.

Let $H$ be a cocommutative bialgebra with a coproduct $\Delta $.
We will use the Sweedler's notation
\[
\Delta(f) = f_{(1)}\otimes f_{(2)}, \quad
(\Delta\otimes \idd)\Delta(f)= (\idd\otimes \Delta)\Delta(f)=
f_{(1)}\otimes f_{(2)}\otimes f_{(3)},\quad \text{etc.}
\]
Define $\Delta^k$, $k\ge 0$, as follows:
$\Delta^0=\idd_H$, $\Delta^{k}=(\idd\otimes \Delta^{k-1})\Delta$,
i.e.,
$\Delta^{k}(f)=f_{(1)}\otimes \dots\otimes f_{(k)}$.

Consider the class of objects
$\mathcal M$ that consists of left modules
over $H$. The following structure turns  $\mathcal M$
into a multicategory called $H$-pseudo-tensor category:
\begin{equation}\label{H-multi}
 P_n^{\mathcal M}(A_1,\dots ,A_n; A)=
\Hom_{H^{\otimes n}}(A_1\otimes \dots \otimes A_n, H^{\otimes n}\otimes_H A)
\end{equation}
(the tensor product $\otimes $ without any index is assumed to be over
the ground field $\Bbbk$).
Here we consider $H^{\otimes k}$ as the outer product of regular right
$H$-modules, i.e.,
\[
F\cdot f = F\Delta^k(f),\quad F\in H^{\otimes k},\ f\in H.
\]

Symmetric groups $S_n$ act on $P_n^{\mathcal M}(\{A_i\};A)$
by permutations of arguments from $A_1,\dots ,A_n$
together with the corresponding permutation of tensor factors
in $H^{\otimes n}$:
\[
 (\sigma f)(a_1,\dots ,a_n)
=
 (\sigma \otimes_H \idd_A)f(a_{\sigma(1)},\dots ,a_{\sigma(n)}),
\]
$f\in P_n^{\mathcal M}(\{A_i\};A)$,
$\sigma \in S_n$, $a_i\in A_i$, $i=1,\dots, n$.
Note that $\sigma $ is an endomorphism of $H$-module
$H^{\otimes n}$ since $H$ is cocommutative.

To define a composition of multimorphisms
$\Comp^\pi$,
$\pi=(m_1,\dots ,m_n)$ is a partition of~$m$,
let us first extend
$f\in P^{\mathcal M}_n(\{A_i\}; A)$ to an $H^{\otimes m}$-linear map
\begin{equation}
f: \bigotimes\limits_{i=1}^n(H^{\otimes m_i}\otimes_H A_i) \to
 H^{\otimes m}\otimes_H A
\end{equation}
as follows:
if
$F_i\in H^{\otimes m_i}$, $a_i\in A_i$, and
\[
f(a_1,\dots ,a_n)=\sum\limits_{j} G_j\otimes _H b_j,\quad
G_j\in H^{\otimes n},
\]
then
\begin{multline}\label{Extended}
f(F_1\otimes_H a_1,\dots, F_n\otimes_H a_n) \\
=
(F_1\otimes\dots\otimes F_n)
 \sum\limits_{j}(\Delta^{m_1-1}\otimes
 \dots\otimes\Delta^{m_n-1})(G_j) \otimes_H b_j .
\end{multline}

Suppose
$\pi=(m_1,\dots ,m_n)$
is a partition,
and consider
$A_{i1},\dots, A_{im_i}\in \mathcal M$,
$B_i \in \mathcal M$, $C\in \mathcal M$,
$f_i\in P^{\mathcal M}_{m_i}(\{A_{ij}\}; B_i)$,
$f\in P^{\mathcal M}_n(\{B_i\}; C)$,
$i=1,\dots, n$.
If
\[
f_i(a_{i1}, \dots, a_{im_i}) =
 \sum\limits_j F_{ij}\otimes _H b_{ij},
 \quad F_{ij}\in H^{\otimes m_i},\ b_{ij}\in B_n,
\]
then define
$\Comp^\pi (f, f_1,\dots, f_n)= f(f_1,\dots,f_n)$
by (\ref{Extended}):
\begin{multline}\label{H-Comp}
f(f_1,\dots, f_n)(a_{11},\dots, a_{1m_1},\dots, a_{n1},\dots, a_{nm_n}) \\
= f(f_1(a_{11},\dots, a_{1m_1}), \dots, f_n(a_{n1},\dots, a_{nm_n}))
\end{multline}

The class $\mathcal M$ with multimorphisms (\ref{H-multi})
and their compositions
(\ref{H-Comp}) is a multicategory denoted by
$\mathcal M^*(H)$ \cite{BDK}.

If $H=\Bbbk$, i.e., $\dim H=1$, then this category is just
the category of linear spaces
$\mathcal V_\Bbbk$ with  polylinear maps as multimorphisms.

Consider the following two functors $F$ and $G$
from $\mathcal M^*(H)$ to $\mathcal M^*(H^{\otimes2})$,
where $H\otimes H$ is considered as the tensor product of
bialgebras:
if $A$ is an $H$-module then
\begin{equation}\label{eq:Func}
F: A \mapsto A\otimes A , \quad
G: A\mapsto H^{\otimes 2}\otimes _H A.
\end{equation}

Indeed, if $f\in P_n^{\mathcal M^*(H)} (A_1,\dots, A_n; A)$,
\begin{gather*}
f(a_1,\dots, a_n) = \sum\limits_j f_{1j}\otimes \dots
 \otimes f_{nj}\otimes_H c_j, \\
f(b_1,\dots, b_n) = \sum\limits_j g_{1j}\otimes \dots
 \otimes g_{nj}\otimes_H d_j,
\end{gather*}
then
\[
F(f)(a_1\otimes b_1, \dots , a_n\otimes b_n)
=\sum\limits_{j,k}
 f_{1j}\otimes g_{1k} \otimes \dots \otimes f_{nj}\otimes g_{nk}
  \otimes_{H^{\otimes 2}} (c_j\otimes d_k).
\]
This expression is well-defined.

In the same way, $G$ acts on multimorphisms as follows:
if
$f\in P_n^{\mathcal M^*(H)} (\{A_i\}; A)$,
\[
f(a_1,\dots, a_n) = \sum\limits_j f_{1j}\otimes \dots \otimes f_{nj}\otimes_H b_j ,
\]
then
\[
G(f)(G_1\otimes_H a_1 , \dots , G_n\otimes_H a_n)
 =
 \sum\limits_{j} G_1\Delta(f_{1j})\otimes \dots \otimes G_n\Delta(f_{nj})
  \otimes_{H^{\otimes 2}} (1\otimes 1\otimes_H b_j),
\]
$G_i\in H^{\otimes 2}$, $f_{ij}\in H$, $b_j\in A$, $i=1,\dots, n$.
This definition is correct since $H$ is cocommutative.
Therefore, $G:A\mapsto H\otimes H\otimes_H A$ is indeed a functor of
multicategories.

\begin{defn}\label{defn41}
An $H$-{\it pseudo-algebra} is an $(F,G)$-pseudo-algebra
in $\mathcal M^*(H)$, where $F$ and $G$ are the functors defined
by~\eqref{eq:Func}.
\end{defn}

This definition is equivalent to the one of \cite{BDK},
where a pseudo-algebra was defined as an $H$-module endowed
with an $(H\otimes H)$-linear map (pseudo-product)
$*: A\otimes A \to H\otimes H\otimes _H A$.

Let $\mathcal C_I $ be the operad from Example~\ref{exmp1}.
Then a $\mathcal C_I$-algebra in $\mathcal M^*(H)$
in the sense of Definition~\ref{defnBD} is also an $H$-pseudo-algebra
since $\mu $ is the image of $\overline{x_1x_2}$.
Conversely, for any $H$-pseudo-algebra $(A,\mu )$
one may construct a functor
$\Phi_A : \mathcal C_\emptyset \to \mathcal M^*(H)$
in such a way that
$\Phi_A(C)=A$,
$\Phi_A(\overline{x_1})=\idd_A$,
$\Phi_A(\overline{x_1x_2}) = \mu $.
(Recall that all multimorphisms of
$\mathcal C_\emptyset $ can be constructed via compositions
and permutations of variables from $\overline{x_1}$
and $\overline{x_1x_2}$.)

In order to show that for $H=\Bbbk[T]$ (with
the canonical Hopf algebra structure)
an $H$-pseudo-algebra is the same as a conformal algebra,
we have to describe the
pseudo-product $\mu :A\otimes A \to H\otimes H\otimes _H A$
in terms of ``usual'' algebraic operations.

Suppose that $H$ is a Hopf algebra with
an antipode~$S$.
One may consider a linear isomorphism
$\Phi: H\otimes H \to H\otimes H$ defined as follows:
\[
\Phi:  f\otimes g \mapsto fS(g_{(1)})\otimes g_{(2)}.
\]
The inverse is easy to find:
$\Phi^{-1}(f\otimes g)= fg_{(1)}\otimes g_{(2)}$.
Therefore, one may choose
an $H$-basis of the product
$H\otimes H$ of regular right modules
in the form
$\{h_{i}\otimes 1\}_{i \in I}$,
where $\{h_i\}_{i\in I} $
is a basis of $H$.
Then we have a well-defined map
\[
\iota : (H\otimes H)\otimes_H A \to H\otimes 1\otimes A\simeq H\otimes A.
\]
Thus, the pseudo-product can be completely described by a family
of binary $\Bbbk $-linear operations
\[
(\cdot\oo{x}\cdot): A\otimes A \to A,
\]
where $x$ ranges over the dual space $X=H^*$.
These operations are defined by
\begin{equation}\label{x-prod}
(a\oo{x}b):=(\langle x, S(\cdot)\rangle\otimes \idd_A)\iota(a*b),
\end{equation}
and satisfy the following axioms:
\begin{itemize}
\item[(H0)] $(a \oo{\alpha x +\beta y}b)=\alpha (a\oo{x}b)+\beta(a\oo{y}b)$;
\item[(H1)] (locality)
$\codim \{x \in X \mid (a \oo{x} b) = 0 \} < \infty$;
\item[(H2)] (sesqui-linearity)
\[
(ha \oo{x} b) = a \oo{xh} b,
\quad
(a \oo{x} hb) = h_{(2)} (a \oo{S(h_{(1)})x} b).
\]
\end{itemize}

An algebraic system obtained on a left $H$-module
$A$ with operations $(\cdot\oo{x}\cdot)$, $x\in X=H^*$,
satisfying the axioms (H0)--(H2)
is  called an $H$-{\it conformal algebra} \cite{BDK}.

In particular, if $H=\Bbbk$ then this is just an ordinary
algebra. If $H=\Bbbk[T]$ then $X\simeq \Bbbk[[t]]$, and it is enough
to consider $x=t^n$, $\langle t^n, T^m\rangle = \delta_{n,m}n!$.
Then
\[
(a\oo{n} b) = (a\oo{t^n}b), \quad n\ge 0,
\]
together with the action of $T$ on $A$ determines the structure
of a conformal algebra on $A$ in the sense of \cite{K1}
(see Introduction). Therefore, a conformal algebra is just
an $H$-pseudo-algebra in the sense of Definition~\ref{defn41}.

The notions of associative (commutative, Lie, etc.)
conformal algebras can also be translated to the language
of pseudo-algebras in a general way.
One approach was proposed
in \cite{Ro1}, using the notion of a coefficient algebra. Namely,
for any conformal algebra $A$ there exists uniquely defined ordinary
algebra $\Coeff A$
(coefficient algebra, or annihilation algebra \cite{K2})
such that $(\Coeff A)[[z,z^{-1}]]$ ``universally
contains" the conformal algebra $A$. If $\mathcal V$ is a variety
of algebras and $\Coeff A \in \mathcal V$ then $A$ is said to
be a $\mathcal V$-conformal algebra.
Such algebras form a category with respect to
homomorphisms of conformal algebras.

Another approach is to apply Definition~\ref{defnBD}.
If $\mathcal V$ is a variety of ordinary algebras, then there
exists a collection $I=I(\mathcal V)$ of
polylinear homogeneous defining identities.
If $A$ is an $H$-pseudo-algebra and the functor
$\Phi_A :\mathcal C_\emptyset \to A$ can be restricted
to $\mathcal C_I$ then it is natural to say that $A$
is a  $\mathcal V$-algebra in $\mathcal M^*(H)$.
The class of all such pseudo-algebras
form a category.

In \cite{Ko} it was shown that the last approach is equivalent
to the one of \cite{Ro1}.

\begin{thm}[\cite{Ko}]\label{thm:var}
Let $\mathcal V$ be a variety of algebras.
A conformal algebra $A$ is a $\mathcal V$-algebra in
$\mathcal M^*(\Bbbk[T])$
if and only if $\Coeff A$ is a $\mathcal V$-algebra.
\end{thm}

\section{TC-algebras}\label{sec4}

In this section, by conformal algebras we mean
$(F,G)$-pseudo-algebras in
$\mathcal M^*(H)$, $H=\Bbbk[T_1,\dots, T_n]$, corresponding
to the functors $F: X\mapsto X\otimes X$,
$G: X\mapsto H\otimes H\otimes_H X$.
The algebra $H$ is considered as a topological algebra with
basic neighborhoods of zero given by powers of the
augmentation ideal
$(T_1,\dots, T_n)$.

Suppose $A$ is a (Hausdorff) topological algebra
endowed with continuous derivations $\partial_1,\dots, \partial_n$.
A map $a: H\to A$ is said to be translation-invariant
(T-invariant, for short) if
\begin{equation}\label{T-inv}
a \dfrac{\partial}{\partial T_i} = \partial _i a, \quad i=1,\dots, n.
\end{equation}
Denote the set of all continuous T-invariant maps from $H$ to $A$
by $\mathcal F(A)$.

\begin{exmp}\label{exmp:Weyl}
(i) Consider $A=\mathbb A_n$, where
\[
\mathbb A_n =\Bbbk\langle p_1,\dots,p_n,
q_1,\dots, q_n \mid [p_i,p_j]=[q_i,q_j]=0,\, [q_i,p_j]=\delta_{i,j}\rangle
\]
(the $n$th Weyl algebra) with respect to
$q$-adic topology, i.e., the system of basic neighborhoods of zero
is given by left ideals
$Q_1 \supset Q_2 \supset \dots$,
\[
Q_k = \sum\limits_{i_1,\dots, i_k=1}^n \mathbb A_n q_{i_1}\dots q_{i_k}, \quad k\ge 1.
\]
Derivations $\partial_i =[\cdot , p_i]$
are continuous, and for any $f\in \Bbbk[p_1,\dots, p_n]$
the map
\[
a_f: T_1^{a_1}\dots T_n^{a_n} \mapsto f q_1^{a_1}\dots q_n^{a_n}
\]
is continuous and T-invariant.

(ii) Consider $W_n\subset \mathbb A_n$,
$W_n=\sum_{i=1}^n\Bbbk[p_1,\dots ,p_n]q_i$,
i.e., the space of all derivations of $H$.
This is a topological Lie algebra with respect to $p$-adic
topology. Derivations $\partial_i =[q_i,\cdot ]$
are continuous, and for any $i=1,\dots, n$
the map
\[
a_i: T_1^{a_1}\dots T_n^{a_n} \mapsto p_1^{a_1}\dots p_n^{a_n}q_i
\]
is continuous and T-invariant.
\end{exmp}

Note that $\mathcal F(A)$
can be considered as an $H$-module with respect to
\[
 ( T_i a )(f) = -a\left (\frac{\partial f}{\partial T_i}\right)
=-\partial _i a(f),
\quad a\in \mathcal F(A),\ f\in H.
\]

If $B$ is a subspace of $A$
then by $\mathcal F(B)$ we denote the space of all
$B$-valued maps from $\mathcal F(A)$.
If $C$ is an $H$-submodule of $\mathcal F(A)$ then by
$\mathcal A(C)$ we denote the subspace
\[
\mathcal A(C) =\{a(f) \mid a\in C,\, f\in H\}.
\]
Note that $B\supseteq \mathcal A(\mathcal F(B))$,
$\mathcal F(\mathcal A(C))\supseteq  C$.

\begin{defn}\label{defn:TC}
A topological algebra $A$ with continuous derivations
$\partial_i$, $i=1,\dots, n$, is said to be a {\it TC-algebra}
if $A=\mathcal A(\mathcal F(A))$.
\end{defn}

For any TC-algebra $A$ the derivations $\partial_i$ necessarily
commute and each of them is locally nilpotent.

For example, the polynomial algebra $H=\Bbbk[T_1,\dots, T_n]$
is an associative TC-algebra, $\idd_H \in \mathcal F(H)$.
If $n$ is even then the same $H$ with respect to
the Poisson bracket
\[
\{f,g\} = \sum\limits_{i=1}^{k}
 \frac{\partial f}{\partial T_{i}}\frac{\partial g}{\partial T_{k+i}}
 - \frac{\partial f}{\partial T_{k+i}}\frac{\partial g}{\partial T_{i}},
\quad n=2k, \ f,g\in H,
\]
is a Lie TC-algebra (as usual, denoted by $P_n$).

The Weyl algebra from Example~\ref{exmp:Weyl} is an associative
TC-algebra.
The Lie algebra $W_n\subset \mathbb A_n^{(-)}$
is also a TC-algebra, as well as its classical
subalgebras $S_n=\{D\in W_n\mid Dv=0\}$ and
$H_n=\{D\in W_n \mid Ds=0\}$,
$v= \mathrm dT_1\wedge \dots \wedge \mathrm dT_n $,
$s = \sum\limits_{i=1}^k \mathrm dT_i\wedge \mathrm dT_{k+i}$,
$n=2k$.

\begin{remark}
Note that $W_n$ is not a TC-subalgebra of $\mathbb A_n^{(-)}$. One
has to consider $p$-adic topology on $W_n$ and set $\partial_i
=[q_i, \cdot]=\frac{\partial}{\partial p_i}$. The same settings
work for $S_n$ and $H_n$. Let us consider $H_n$ in details. It is
well-known that $H_n$ is a homomorphic image of $P_n$: $f\mapsto
\{f,\cdot\}\in H_n$, $f\in H$. It is obvious that the map $a: H\to
H_n$, $a(f)= \{f,\cdot\}$ is continuous and T-invariant.
\end{remark}

TC-algebras form a category where morphisms are continuous homomorphisms
of algebras commuting with $\partial_i$, $i=1,\dots, n$.
By an ideal of a TC-algebra $A$ we mean
a $\partial_i$-invariant ideal of~$A$.

\begin{prop}
{\rm (i)} If $A$ is a TC-algebra then the matrix algebra
$\mathbb M_N(A)$ is also a TC-algebra.

{\rm (ii)} If $A$ is an associative TC-algebra and
$\sigma :A\to A$ is a continuous $\partial_i$-invariant
involution then $\mathrm{Sym}(A,\sigma)$ and
$\mathrm{Skew}(A,\sigma)$ are Jordan and Lie TC-algebras,
respectively.

{\rm (iii)} If $A$ is a TC-algebra for $H'=\Bbbk[T_1,\dots, T_r]$,
$n\ge r$, then the polynomial algebra
$A[T_{r+1},\dots , T_{n}]$ is a TC-algebra for
$H=\Bbbk[T_1,\dots , T_n]$.
\end{prop}

\begin{proof}
(i) Suppose
$x=(a_{ij}(f_{ij})) \in \mathbb M_N(A)$.
Then $x$ can be presented as
$x=(b_{ij} (T^\alpha))$
for an appropriate
$\alpha =(\alpha_1,\dots,\alpha_n)\in \mathbb Z_+^n$,
$T^\alpha =T_1^{\alpha_1}\dots T_n^{\alpha_n}$.
Hence, $x=b(T^\alpha)$, $b=(b_{ij})\in \mathbb M_N(\mathcal F(A))$.

(ii) Just note that for every $a\in \mathcal F(A)$
the maps $f\mapsto a(f)\pm \sigma a(f)$, $f\in H$, also
belong to $\mathcal F(A)$.

(iii) Define additional derivations and topology on $A'=A[T_{r+1},\dots, T_n]$
in the usual way. Then for any continuous T-invariant map
$a:H'\to A$ the map
$a': H\to A'$ given by
\[
a'(T_1^{a_1}\dots  T_n^{a_n}) = a(T_1^{a_1}\dots T_r^{a_r})T_{r+1}^{a_{r+1}}\dots T_n^{a_n}
\]
is also continuous and T-invariant.
\end{proof}

The problem of studying a TC-algebra $A$ can be reduced
to the space $\mathcal F(A)$. The last object is in
some sense smaller,
but the algebraic structure on
$\mathcal F(A)$ is more complicated.

One may define a family of bilinear operations
\[
(\cdot_{(f)}\cdot): \mathcal F(A)\otimes \mathcal F(A) \to \mathcal F(A),
\quad f\in H,
\]
in the following way:
\begin{equation}\label{eq:f-product}
(a_{(f)} b)(g) =
 a(f_{(1)})b(S(f_{(2)})g),\quad a,b\in \mathcal F(A),\ f,g \in H,
\end{equation}
$S$ stands for the standard antipode in $H$.
It is easy to check that $T_i a\in \mathcal F(A)$,
$(a_{(f)}b)\in \mathcal F(A)$ for every
$a,b \in \mathcal F(A)$.

Identify the space of polynomials $H$ with a subalgebra of $H^*$
(getting product in $H^*$ as dual to the coproduct in $H$) in the
natural way: $T_i\mapsto T_i^*$, where $\langle T_i^*, T_j\rangle
=\delta_{i,j}$. Then $\mathcal F(A)$ turns into an $H$-module
endowed with bilinear operations indexed by $H^*$. It is
straightforward to show that (H0) and (H2) hold (cf. \cite{BDK}).
However, (H1) does not hold in general. But if $C$ is an
$H$-submodule of $\mathcal F(A)$ that consists of elements
satisfying (H1), and $C$ is closed under all operations
$(\cdot\oo{f}\cdot)$, $f\in H$, then $C$ is a conformal algebra.

One may interpret Theorem~\ref{thm:var} as follows.

\begin{thm}
Suppose a TC-algebra $A$ belongs to a variety $\mathcal V$.
Then any conformal algebra $C\subseteq \mathcal F(A)$
is a $\mathcal V$-algebra in $\mathcal M^*(H)$.
\end{thm}

\begin{proof}
It is sufficient to perform the same computations as in
\cite{Ko} for multi-dimen\-si\-onal case.
\end{proof}

There is
a sufficient condition for entire $\mathcal F(A)$ to satisfy (H1),
i.e., to be a conformal algebra.
Now $A$ is supposed to be associative, Lie, or Jordan.

Let $M$ be an $H$-module,
$E=\End_{\Bbbk} M$. One may consider $E$ as a topological algebra
with respect to finite topology \cite{J}, i.e., such that
basis of neighborhoods of zero is given by subspaces
\[
U_{u_1,\dots,u_N} = \{\psi \in E\mid \psi u_j =0,\, j=1,\dots,N\},
\quad u_1,\dots, u_N\in M.
\]
Then
$\partial _i  = [\cdot, T_i]\in \Derr E$,
$i=1,\dots, n$, are continuous derivations.

\begin{defn}\label{defn:TC-rep}
A TC-representation of a TC-algebra $A$
is a continuous $\partial_i$-invariant
representation $\rho : A \to \End_{\Bbbk} M$.
In this case, $M$ is said to be a TC-module over~$A$.
\end{defn}

For example, the space $H$ itself with respect to the
canonical action of the Weyl algebra
can be considered as a TC-module
over $\mathbb A_n$.

\begin{prop}[cf. \cite{BDK}]\label{prop:conf}
Let $A$ be a TC-algebra.
If $A$ has a faithful TC-represen\-ta\-tion $\rho $ on a
finitely generated $H$-module $M$ then $\mathcal F(A)$
is a conformal algebra.
\end{prop}

\begin{proof}
Suppose $M$ is generated over $H$ by elements
$u_1,\dots, u_N \in M$.
It follows from the definition of finite topology
that
$\dim \rho(a(H))u_i<\infty$ for every $a\in \mathcal F(A)$,
$i=1,\dots, n$. Moreover, there exists
$s\ge 1$ such that $\rho(a(I^s))u_i =0$ for all~$i$,
where $I$ is the augmentation ideal of~$H$.

Relation \eqref{eq:f-product} implies that for every
$a,b\in \mathcal F(A)$, $f\in I^K$, $K\ge 1$,
we have
\[
(a\oo{f}b)(g) \in \sum\limits_{k=0}^K a(I^k)b(I^{K-k}),
\quad g\in H,
\]
where $I^0=\Bbbk $.

Let $s$ stands for a number such that
$\rho(b(I^{s}))u_i=0$ for all $i=1,\dots, n$.
Denote $V=\sum_{i=1}^n \rho(b(H))u_i \subset M$.
Since $\dim V<\infty$, there exists
$m\ge 1$ such that
$\rho(a(I^{m}))V =0$.
If $K_1>s+m$ then $\rho(a(I^k)b(I^{K_1-k}))u_i=0$
for all $k=0,\dots, K_1$, $i=1,\dots, n$.
In the same way, we can find $K_2$ such that
$\rho(b(I^{K_2-k})a(I^{k}))u_i=0$
for all $k=0,\dots, K_2$, $i=1,\dots, n$.
Since $\rho $ is a representation of $A$
(recall that $A$ is either associative, or Lie, or Jordan),
we obtain $\rho((a\oo{f} b)(g))u_i = 0$ for all $f\in I^K$,
$K\ge \max\{K_1,K_2\}$, $g\in H$.
It remains to note that $\partial_i$-invariance of $\rho $
implies $(a\oo{f} b)=0$.
\end{proof}

\begin{cor}[cf. \cite{K2}]
The following objects are associative conformal algebras:
\begin{itemize}
\item $\mathcal F(\mathbb M_N(H))=\Curr^n_N$ (current conformal algebra
over $\mathbb M_N(\Bbbk )$);
\item $\mathcal F(\mathbb A_n) =\Cend^n_1$ (conformal Weyl algebra);
\item $\mathcal F(\mathbb M_N(\mathbb A_n))=\Cend^n_N$
 (algebra of conformal endomorphisms of a free $N$-generated $H$-module);
\item $\mathcal F(\mathbb A_{n,N,Q}) = \Cend^n_{N,Q}$,
where $\mathbb A_{n,N,Q}=\mathbb M_N(\mathbb A_n)Q(p_1,\dots, p_n)$,
 $Q$ is a matrix over $\Bbbk [p_1,\dots, p_n]$.
\end{itemize}
\end{cor}

The last conformal algebra appears from TC-subalgebra
$\mathbb A_{n,N,Q} $
of the matrix Weyl algebra. If $\det Q\ne 0$ then $\Cend_{N,Q}^n$
is simple \cite{K2}.

By abuse of terminology, let us call a TC-algebra $A$
satisfying the condition of Proposition~\ref{prop:conf}
by a TC-algebra with finite faithful representation.
If, in addition, $\mathcal F(A)$ is a finitely generated $H$-module
then $A$ is said to be finite TC-algebra. A~conformal algebra
which is finitely generated over $H$ is also called finite.

The structure of finite simple and semisimple
Lie conformal algebras was completely
described in \cite{DK} ($n=1$) and \cite{BDK} ($n\ge 1$).
Finite Jordan conformal algebras were considered in
\cite{Z1} ($n=1$), where simple and semisimple algebras were described.
Our aim is to present a similar result for the general class
of not necessarily finite associative
TC-algebras with finite faithful representation.
The complete solution is known for $n=1$
(for $n=0$ this is the classical Wedderburn Theorem).
Throughout the rest of the paper we consider associative algebras
only.

\begin{prop}\label{prop:semisimple}
Let $A$ be a semiprime (i.e., without nonzero nilpotent ideals)
TC-algebra with finite faithful representation.
Then $\mathcal F(A)$ is a semiprime conformal algebra.
\end{prop}

\begin{proof}
Relation
\eqref{eq:f-product} implies that
\[
a(f)b(g) = (a\oo{f_{(1)}}b)(f_{(2)}g), \quad a,b\in \mathcal F(A),
\ f,g\in H.
\]
Therefore, if $\mathcal F(A)$ has a nonzero abelian ideal $J$
then
$\{a(f)\mid f\in H,a\in J\}$
is a nonzero abelian ideal of~$A$.
\end{proof}

\begin{thm}\label{thm:irred}
Let $A$ be a TC-subalgebra of $\mathbb M_N(\mathbb A_n)$.
If the subalgebra $A_1 = \Bbbk[p_1,\dots, p_n]A$ acts
irreducibly on $M=H\otimes \Bbbk^N$ then
$A_1$ is a left ideal in $\mathbb M_N(\mathbb A_n)$
which is dense with respect to the $q$-adic topology.
\end{thm}

\begin{proof}
Let us first note that $A_1$ is a subalgebra.
Indeed, $ap_i = p_i a + \partial _i (a)$, $a\in A$, $i=1,\dots, n$.
Hence, $A_1A_1\subseteq A_1$.

The centralizer $\mathcal D = \End_{A_1} M$ of $A_1$ in $\End M$
is a division algebra. Consider
$0\ne \varphi \in \mathcal D$, $0\ne a\in A$, $f\in H$.
Then $fa = f(p_1,\dots, p_n)a \in A_1$, and we have
\[
[\varphi, f]a = \varphi (fa) - f(\varphi a) = 0.
\]
Note that
$[\varphi, f] \in \mathcal D$ since for every $b\in A_1$
the commutator $[f,b]$ belongs to $A_1$.
Therefore, for every $f\in H$, $\varphi \in \mathcal D$
we have $\varphi f = f\varphi$, so $\mathcal D \subseteq \mathbb M_N(H)$.
Identity endomorphism of $\End M$ belongs to $\mathcal D$.
But it is clear that $\mathbb M_N(H)$ contains no division
algebra with the same identity except for $\mathcal D =\Bbbk $.
Hence, $A_1$ is a dense subalgebra of
$\mathbb M_N(\mathbb A_n)\subset \End M$
with respect to the finite topology on $\End M$ (over $\Bbbk $).
This topology is actually equivalent to the $q$-adic topology.

Now, consider $\mathcal F(A_1)\supseteq \mathcal F(A)=C$.
Since every TC-algebra $B$ acts on $\mathcal F(B)$ as on a
left module by the rule
\[
a(f)\cdot b = (a\oo{f} b), \quad a,b\in \mathcal F(B),\ f\in H,
\]
we have $A_1\cdot \mathcal F(A_1) \subseteq \mathcal F(A_1)$.
Since $A_1$ is dense, $\mathcal F(A_1)$ is a left ideal of
$\Cend^n_N = \mathcal F(\mathbb M_N(\mathbb A_n))$.
It follows from \eqref{eq:f-product} that $A_1$ is a left ideal
of $\mathbb M_N(\mathbb A_n)$.
\end{proof}

For $n=1$, the last theorem allows to deduce what is a structure
of simple and semisimple conformal algebras with finite
faithful representation \cite{Ko2}.
In the language of TC-algebras, these results may be stated as
follows.

\begin{thm}
Let $A$ be a TC-algebra over $H=\Bbbk[T]$
with finite faithful representation.

{\rm (i)} If $A$ is simple then $A\simeq \mathbb M_N(H)$
or $A\simeq \mathbb A_{1,N,Q}$, $\det Q\ne 0$.

{\rm (ii)} If $A$ is semiprime then $A$ is isomorphic
to a finite direct sum of simple ones from~(i).
\end{thm}

\begin{proof}
By Proposition \ref{prop:semisimple}, if $A$ is semiprime
then $C=\mathcal F(A)$ is semiprime.
In \cite{Ko2} it was shown that a semiprime conformal algebra
$C$
with finite faithful representation is isomorphic to
a direct sum of algebras $C_i$, $C_i\simeq \Curr _{N_i} $ or
 $C_i\simeq \Cend_{N_i,Q_i}$, $\det Q_i\ne 0$.
Therefore, $A=\mathcal A(C) = \mathcal A\bigg(
\bigoplus_{i} C_i \bigg) = \bigoplus_{i} \mathcal A(C_i)$,
that proves~(ii).
Note that $A$ is even semisimple.
Statement (i) is now obvious.
\end{proof}

\begin{thm}
If $A$ is a TC-algebra with finite faithful representation
then its Jacobson radical is nilpotent.
\end{thm}

\begin{proof}
Consider $C=\mathcal F(A)$.
It was shown in \cite{Ko2} that an associative conformal algebra
with finite faithful representation has a maximal nilpotent
ideal (radical).

Suppose $R$ is the radical of $C$.
It is easy to show that $C/I$ is a conformal algebra with
finite faithful representation.
Then
$\mathcal A(R)\subseteq A$ is a nilpotent ideal, and
$A/\mathcal A(I)$ is isomorphic to $ \mathcal A(C/I)$
which is semisimple. Therefore, $J(A)=\mathcal A(R)$.
\end{proof}

\section{Open problems}

{\bf 1.} Describe irreducible conformal
subalgebras of $\Cend_N^n$ for $n>1$,
i.e., those $C$ that $\Bbbk[p_1,\dots, p_n]\mathcal A(C)$
acts irreducibly on $M=H\otimes \Bbbk^N$.

This is equivalent to the problem of description all
TC-subalgebras $A\subseteq \mathbb M_N(\mathbb A_n)$
such that $\Bbbk[p_1,\dots, p_n]A$ acts irreducibly
on $M=H\otimes \Bbbk^N$.
Theorem \ref{thm:irred} provides an important property
of such algebras, however, the complete description
is obtained only for $n=1$ \cite{Ko2} (for $n=0$ this is the
classical Burnside's theorem).
For $n=1$, the corresponding TC-algebras are (up to automorphism)
$\mathbb M_N(H)$ and $W_{N,Q}=\mathbb M_N(\mathbb A_1)Q(p)$,
$\det Q\ne 0$.

\medskip
{\bf 2.} Describe irreducible Lie conformal subalgebras
of $\mathrm{gc}^n_N = \mathcal F(\mathrm{gl}(H\otimes \Bbbk^N))$,
$H=\Bbbk[T_1,\dots, T_n]$.

A great advance in this problem was obtained in \cite{DSK}
and in \cite{Z2}, but there is no complete solution even for $N=n=1$.
In \cite{BKL}, the following conjecture was stated:
if $C$ is an infinite irreducible Lie conformal subalgebra
of $\mathrm{gc_N^1}$ then the corresponding
TC-algebra $L=\mathcal A(C)$ is isomorphic
either to $W_{N,Q}^{(-)}$ or to $\mathrm{Skew}(W_{N,Q}, \sigma)$
for an appropriate TC-involution $\sigma $ of $W_{N,Q}$,
$\det Q\ne 0$.

\end{document}